\documentclass[reqno, 11pt, letterpaper, oneside]{article} 

\usepackage{amsfonts,amsmath,amsthm}
\usepackage{geometry}

\newtheorem{theorem}{Theorem}

\newtheorem{corollary}[theorem]{Corollary}
\newtheorem{conjecture}[theorem]{Conjecture} 

\newcommand{\cE}{{\mathcal E}}

\newcommand{\cH}{{\mathcal H}}

\newcommand{\PP}{{\mathbb P}}

\renewcommand{\epsilon}{\varepsilon}
\newcommand\arxiv[1]{\texttt{\def~{{\tiny$\sim$}}{arXiv:#1}}} 

\long\def\symbolfootnote[#1]#2{\begingroup
\def\thefootnote{\fnsymbol{footnote}}\footnote[#1]{#2}\endgroup}

\begin{document}

\begin{center}

{\LARGE Dense subgraphs in the $H$-free process}

\vspace{3mm}

{\large Lutz Warnke\symbolfootnote[2]{The author was supported by a Scatcherd European Scholarship and an EPSRC Research Studentship.}}
\vspace{1mm}

{ Mathematical Institute,  University of Oxford\\
 24--29 St.~Giles', Oxford OX1 3LB, UK\\ 
 {\small\tt warnke@maths.ox.ac.uk}}

\vspace{6mm}

\small

\begin{minipage}{0.8\linewidth}
\textsc{Abstract.} 
The $H$-free process starts with the empty graph on $n$ vertices and adds edges chosen uniformly at random, one at a time, subject to the condition that no copy of $H$ is created, where $H$ is some fixed graph. 
When $H$ is strictly $2$-balanced, we show that for some $c,d> 0$, with high probability as $n \to \infty$, the final graph of the $H$-free process contains no subgraphs $F$ on $v_F \leq n^{d}$ vertices with maximum density $\max_{J \subseteq F}\{e_J/v_J\} \geq c$. 
This extends and generalizes results of Gerke and Makai for the $C_3$-free process. 
\end{minipage}
\end{center}  

\normalsize

\section{Introduction}%
Almost fifty years ago, Erd{\H{o}}s and R{\'e}nyi~\cite{ErdosRenyi1959} introduced the \emph{random graph process} $G(n,i)$. 
This starts with the empty graph on $n$ vertices and adds $i$ edges one by one, where each edge is chosen uniformly at random among all edges not yet present. 
In their seminal 1960 paper~\cite{ErdosRenyi1960}, the first problem they studied is the so-called \emph{small subgraphs problem}. 
Given a fixed graph $F$ with $e_F$ edges and $v_F$ vertices, it asks whether $G(n,i)$ whp\footnote{As usual, we say that an event holds \emph{with high probability}, or \emph{whp}, if it holds with probability $1-o(1)$ as $n\to\infty$.} contains a copy of $F$ as a subgraph or not. 
It took twenty years until Bollob{\'a}s~\cite{bollobas81} solved this problem in full generality, showing that the so-called \emph{maximum density} $m(F)$ is the crucial parameter that essentially determines the appearance of $F$. 
\begin{theorem}%
\label{thm:small:subgraphs}%
{\normalfont\cite{bollobas81}} 
Let $F$ be a fixed non-empty graph. Then 
\[
\lim_{n \to \infty} \PP[F \subseteq G(n,i)] = \begin{cases}
0 & \text{ if $i = o\big(n^{2-1/m(F)}\big)$}\\
1 & \text{ if $i = \omega\big(n^{2-1/m(F)}\big)$} 
\end{cases} \enspace ,
\]
where $m(F) := \max\{ e_J/v_J \mid J \subseteq F \text{ and } v_J \geq 1 \}$. 
\end{theorem}
In a related model it took nearly thirty years to solve the small subgraphs problem. 
Random $d$-regular graphs $G_{n,d}$ have been studied since around 1980, cf.~\cite{Wormald1999}, but only in 2007 Kim, Sudakov and Vu~\cite{KimSudakovVu2007} established the analogue of Theorem~\ref{thm:small:subgraphs} for $G_{n,d}$;  again $m(F)$ turns out to be the important quantity.

In this paper we consider a natural variant of the classical random graph process, which has recently attracted a lot of attention. 
Given some fixed \mbox{graph $H$}, this process also starts with an empty graph and then add edges one by one, but each new edge is now chosen uniformly at random subject to the condition that no copy of $H$ is formed. 
This so-called \emph{$H$-free process} was suggested by Bollob{\'a}s and Erd{\H o}s~\cite{Bollobas2010PC} at a conference in 1990, and it was first described in print in 1995 by Erd{\H o}s, Suen and Winkler~\cite{ErdoesSuenWinkler1995}, who asked how many edges the final graph typically has. 
In 2001 Osthus and Taraz~\cite{OsthusTaraz2001} answered this basic question up to logarithmic factors for the class of \emph{strictly $2$-balanced} graphs, i.e.\ where $H$ satisfies $v_H,e_H \geq 3$ and for all proper subgraphs $K$ of $H$ with $v_K \geq 3$ vertices we have 
\[ \frac{e_K-1}{v_K-2} < \frac{e_H-1}{v_H-2} =:  d_2(H) \enspace . \] 
Many interesting graphs are strictly $2$-balanced, including cycles $C_{\ell}$, complete graphs $K_s$, complete $r$-partite graphs $K_{t,\ldots,t}$ and the $d$-dimensional cube.
Only in a breakthrough in 2009, Bohman~\cite{Bohman2009K3} was able to close the logarithmic gap mentioned above for the $C_3$-free process, giving (up to constants) matching upper and lower bounds for the final number of edges. 
Subsequently several additional special cases have been settled, see e.g.~\cite{Picollelli2010C4, Picollelli2010K4Minus, Warnke2010K4, Wolfovitz2010K4}. 
Very recently, the final number of edges in the $C_{\ell}$-free process was resolved in~\cite{Warnke2010Cl}, giving the first matching bounds for a non-trivial class of graphs. 
As one can see, much research has been devoted to understanding the combinatorial structure of the final graph of the $H$-free process, but so far even very basic properties  are not well understood.

The main focus of the present work is the small subgraphs problem in the final graph of the $H$-free process, where $H$ is strictly $2$-balanced. 
An intriguing consequence of the recent analysis of Bohman and Keevash~\cite{BohmanKeevash2010H} is 
that although the graph $G(i)$ produced by the $H$-free process  after adding $i$ edges contains no copies of $H$, during some initial phase the number of small subgraphs in $G(i)$ and the unconstrained $G(n,i)$ are roughly the same 
(for the $C_3$-free process similar results were obtained by Wolfovitz~\cite{Wolfovitz2009K3Cycles,Wolfovitz2009K3Subgraph}). 
Regarding subgraph containment, their results imply the following statement. 
Recall that $m(F)$ denotes the maximum density of a graph $F$. 
\begin{theorem}%
\label{thm:small:subgraphs:H-free}%
{\normalfont\cite{BohmanKeevash2010H}} 
Let $H$ be a strictly $2$-balanced graph. Suppose that $F$ is a fixed non-empty $H$-free graph. Then there exists $\xi = \xi(F,H)$ such that for $m := \xi n^{2-1/d_2(H)}(\log n)^{1/(e_H-1)}$ we have 
\[
\lim_{n \to \infty} \PP[F \subseteq G(m)] = \begin{cases}
0 & \text{ if $m(F)> d_2(H)$}\\
1 & \text{ if $m(F) \leq d_2(H)$} 
\end{cases} \enspace . 
\]
\end{theorem}
Unfortunately, the results in~\cite{BohmanKeevash2010H} only hold during the first $m$ steps of the $H$-free process, which motivates further investigation of the evolution in later steps. 
Since precise structural properties of the final graph are not known so far, we do not ask whether an analogue of Theorem~\ref{thm:small:subgraphs} also holds in the later evolution, but restrict our attention to the basic question whether fixed $H$-free graphs $F$ satisfying, say, $m(F) \gg d_2(H)$ appear or not. 
For the special case $H=C_3$ this has recently been addressed by Gerke and Makai~\cite{GerkeMakai2010K3}: they proved that for some $c>0$, whp fixed graphs $F$ with $e_F/v_F \geq c$ do not appear in the $C_3$-free process. 
We would like to remark that such a behaviour is not necessarily true for all constrained graph processes. 
For example,  Gerke  Schlatter, Steger and Taraz~\cite{GerkeSchlatterStegerTaraz2008} showed that in the random planar graph process (where random edges are added subject to the condition of maintaining planarity) \emph{every} fixed planar graph appears whp. 
This raises the question what behaviour the general $H$-free process exhibits.

\subsection{Main result}%
In this paper we prove that for strictly $2$-balanced $H$, the $H$-free process contains whp no copies of sufficiently dense graphs, even if their sizes grow moderately in $n$. 
In fact, we obtain a new result for the \emph{final} graph of the $H$-free process regarding the number of edges in every small subset of the vertices.  
As usual, given $A \subseteq [n]$, we write $e(A)$ for the number of edges joining vertices in $A$. 
\begin{theorem}%
\label{main-theorem}%
For every strictly $2$-balanced graph $H$ there exist $c,d> 0$ such that whp in the final graph of the $H$-free process we have $e(A) < c |A|$ for all $A \subseteq [n]$ with $1 \leq |A| \leq n^{d}$. 
\end{theorem}
This estimate is in contrast to most known results, which only hold during some initial number of steps. 
We immediately deduce the following statement regarding the small subgraphs problem in the final graph of the $H$-free process, which complements the results in~\cite{BohmanKeevash2010H} (see e.g.\ Theorem~\ref{thm:small:subgraphs:H-free}). 
\begin{corollary}%
\label{main-corollary}%
For every strictly $2$-balanced graph $H$ there exist $c,d> 0$ such that whp the final graph of the $H$-free process contains no copy of any graph $F$ with $1 \leq v_F \leq n^{d}$ vertices and $m(F) \geq c$. \qed 
\end{corollary}%
Up to constants the bound $m(F) \geq c$ is best possible, since the results of Bohman and Keevash~\cite{BohmanKeevash2010H} imply that $H$-free graphs $F$ with $v_F = O(1)$ vertices and $m(F) \leq d_2(H)$ do appear in the $H$-free process (cf.\ Theorem~\ref{thm:small:subgraphs:H-free}). 
For the special case $H=C_3$, Gerke and Makai~\cite{GerkeMakai2010K3} have previously obtained a similar result for fixed graphs $F$. 
So, Corollary~\ref{main-corollary} not only generalizes the main result of~\cite{GerkeMakai2010K3}, but moreover demonstrates that whp dense graphs $F$ never appear in the $H$-free process, also if their number of vertices grow moderately \mbox{in $n$}. 
In fact, we believe that fixed graphs with maximum density strictly larger than $d_2(H)$ do not appear in the $H$-free process. 
\begin{conjecture}%
Let $H$ be a strictly $2$-balanced graph and suppose that $F$ is a fixed non-empty graph satisfying $m(F) > d_2(H)$. 
Then whp the final graph of the $H$-free process contains no copy of $F$. 
\end{conjecture}
We now outline our strategy for proving Theorem~\ref{main-theorem}.  
Intuitively, we show that whp for \emph{every} possible placement of $\lceil c|A| \rceil$ edges inside some set $A  \subseteq [n]$ satisfying $|A| \leq n^d$, already after the first $m$ steps there exists a `witness' which certifies that not all of these edges can appear in the $H$-free process.    
The same basic idea was used in \cite{GerkeMakai2010K3}, but the main part of their argument is tailored towards the (simpler) $C_3$-free case (in fact, a similar idea has also previously been used for bounding the independence number of the $H$-free process in \cite{Bohman2009K3,BohmanKeevash2010H}). 
By contrast, our argument is for the more general $H$-free process, where $H$ is strictly $2$-balanced, and one important ingredient are the estimates obtained by Bohman and Keevash~\cite{BohmanKeevash2010H}. 
For the sake of simplicity and clarity of presentation, we have made no attempt to optimize the constants obtained in our proof, and we also omit floor and ceiling signs whenever these are not crucial.

\section{Preliminaries} 
\label{sec:preliminaries}
In this section we introduce our notation and review properties of the $H$-free process. 
We closely follow \cite{BohmanKeevash2010H}, and the reader familiar with these results may wish to skip this section.

\subsection{Constants, functions and parameters}  
\label{sec:preliminaries:constants:functions:parameters}
In the remainder of this paper we consider a fixed strictly $2$-balanced graph $H$. 
We first choose $\epsilon$ and then $\mu$ small enough such that, in addition to the implicit constraints in \cite{BohmanKeevash2010H} for $H$, we have 
\begin{equation}
\label{eq:def:eps:mu}
\epsilon < \min\left\{\frac{1}{e_H},\frac{1}{2d_2(H)}\right\} \qquad \text{ and } \qquad 2e_{H}(2\mu)^{e_{H}-1} \leq \epsilon \enspace .
\end{equation}
So $\epsilon$ and $\mu$ are absolute constants (depending only on $H$), since the additional constraints in \cite{BohmanKeevash2010H} only depend on $H$. 
Writing $\mathrm{aut}(H)$ for the number of automorphisms of $H$, following \cite{BohmanKeevash2010H} we define 
\begin{equation}
\label{eq:parameters:functions}
p := n^{-1/d_2(H)} \enspace , \quad  m := \mu n^2 p (\log n)^{1/(e_H-1)} \quad \text{ and } \quad q(t) := e^{-2e_{H}{\mathrm{aut}(H)}^{-1}(2t)^{e_H-1}}  \enspace .
\end{equation} 
For every $i \leq m$ we set $t = t(i) := i/(n^2p)$, but will just write $t$ if there is no danger of confusion.

\subsection{Terminology and notation}  
\label{sec:preliminaries:terminology:notation}
Let $G(i)$ denote the graph with vertex set $[n]=\{1,\ldots, n\}$ after $i$ steps of the $H$-free process. 
Its edge set $E(i)$ contains $i$ edges and we partition the remaining non-edges $\binom{[n]}{2} \setminus E(i)$ into two sets $O(i)$ and $C(i)$ which we call \emph{open} and \emph{closed} pairs, respectively. 
We say that a pair $uv$ of vertices is \emph{closed} in $G(i)$ if $G(i) \cup \{uv\}$ contains a copy of $H$. 
Observe that the $H$-free process always chooses the next edge $e_{i+1}$ uniformly at random from $O(i)$. 
In addition, for $uv \in O(i)$ we write $C_{uv}(i)$ for the set of pairs $xy \in O(i)$ such that adding $uv$ and $xy$ to $G(i)$ creates a copy of $H$ containing both $uv$ and $xy$. 
Note that $uv \in O(i)$ would become closed, i.e.\ belong to $C(i+1)$, if $e_{i+1} \in C_{uv}(i)$. 
We remark that in contrast to \cite{BohmanKeevash2010H} we work only with sets of unordered pairs.

\subsection{Previous results for the $H$-free process} 
\label{sec:preliminaries:previous:results:H-free}
Using Wormald's differential equation method~\cite{Wormald1995DEM,Wormald1999DEM}, Bohman and Keevash~\cite{BohmanKeevash2010H} track a wide range of variables throughout the \mbox{first $m$ steps} of the $H$-free process, where $H$ is strictly $2$-balanced. 
From this they deduce their remarkable lower bound on the final number of edges. 
For our argument the key properties are estimates for the number of open and closed pairs and certain density statements. 
The following theorem conveniently summarizes these in a (highly) simplified form. 
\begin{theorem}%
\label{thm:BohmanKeevash2010H}%
{\normalfont\cite{BohmanKeevash2010H}}
Suppose $H$ is strictly $2$-balanced. Set $\beta_H:=e_H(e_H-1)/\mathrm{aut}(H)$.  
Let $\cH_j$ denote the event that for every $n^{2}p \leq i \leq j$, in $G(i)$ we have 
\begin{align}
\label{eq:open-estimate}
|O(i)| & \leq q(t)n^2 \enspace ,\\
\label{eq:closed-estimate}
|C_{uv}(i)| &\geq \beta_H \cdot (2t)^{e_H-2}q(t) p^{-1} && \text{for all pairs $uv \in O(i)$} \enspace , \\ 
\label{eq:closed-intersection-estimate}
|C_{uv}(i)  \cap C_{u'v'}(i)| &\leq  n^{-1/e_H} p^{-1} && \text{for all distinct $uv,u'v' \in O(i)$ and}\\
\label{eq:edges-estimate}
e(A) &\leq \max\big\{8\epsilon^{-1}|A|,p{|A|}^2n^{2\epsilon}\big\} && \text{for all sets $A \subseteq [n]$} \enspace .
\end{align}
Then $\cH_m$ holds whp in the $H$-free process. \qed 
\end{theorem}
Both \eqref{eq:open-estimate} and \eqref{eq:edges-estimate} follow readily from Theorem~$1.4$ and Lemma~$4.2$ in \cite{BohmanKeevash2010H}. 
Corollary~$6.2$ and Lemma~$8.4$  in \cite{BohmanKeevash2010H} give \eqref{eq:closed-estimate} and \eqref{eq:closed-intersection-estimate} by elementary considerations (the `high probability events' in \cite{BohmanKeevash2010H} fail with probability at most $n^{-\omega(1)}$, so there is no problem in taking a union bound over all steps and pairs).

It should be noted that Theorem~\ref{thm:BohmanKeevash2010H} does not directly imply our main result. 
The important difference here is that \eqref{eq:edges-estimate} is only valid during the first $m$ steps, whereas Theorem~\ref{main-theorem} holds in the final graph of the $H$-free process. 
In fact, the proof used in \cite{BohmanKeevash2010H} breaks down when $m$ is too large, and this explains why a different approach is needed to obtain results that also hold in later steps.

\section{The proof}%
\label{sec:proof}%
Recall that when a pair becomes closed it has not yet been added to the graph produced by the $H$-free process, and furthermore can never be added in future steps, as this would create a copy \mbox{of $H$}. 
So, to prove that a certain set of edges $F$ does not appear in the $H$-free process, it suffices to show that already after the first $m$ steps, at least one of its edges is closed. 
\begin{proof}[Proof of Theorem~\ref{main-theorem}]%
For the sake of concreteness we prove the theorem with 
\begin{equation}
\label{eq:def:c:d}
c := \max\left\{\frac{16}{\epsilon},\frac{13 \cdot 2^5}{\beta_H \mu^{e_H-1}}\right\} \qquad \text{ and } \qquad d := \min\left\{\frac{1}{c},\frac{1}{e_H}-\epsilon,\frac{1}{d_2(H)}-2\epsilon,1\right\} \enspace , 
\end{equation}
where $\epsilon$ and $\mu$ are chosen as in Section~\ref{sec:preliminaries:constants:functions:parameters} and $\beta_H$ is defined as in Theorem~\ref{thm:BohmanKeevash2010H}. 
Given $i \leq m$ and $F \subseteq \binom{[n]}{2}$, by $\cE_{F,i}$ we denote the event that $F \cap C(i) = \emptyset$. 
Let $\cE_m$ denote the event that there exists $A \subseteq [n]$ and $F \subseteq \binom{A}{2}$ with $1 \leq |A| \leq n^{d}$ and $|F| = \left\lceil c |A| \right\rceil$ for which $\cE_{F,m}$ holds. 
Note that if $\cE_m$ fails, then $e(A) < c |A|$ for all $A \subseteq [n]$ with $1 \leq |A| \leq n^{d}$, since, as discussed above, none of the corresponding edge sets $F$ can appear in the $H$-free process. 
So, because $\cH_m$ holds whp by Theorem~\ref{thm:BohmanKeevash2010H}, in order to complete the proof it suffices to show 
\begin{equation}
\label{eq:prob_closed_small}
\PP[\cE_m \wedge \cH_m] = o(1) \enspace .
\end{equation}

Fix $A \subseteq [n]$ and $F \subseteq \binom{A}{2}$ with $1 \leq |A|=a\leq n^{d}$ and $|F| = \left\lceil c a \right\rceil$. 
With foresight, let $O_{F}(i) \subseteq O(i)$ denote the open pairs which would close at least one pair of $F$ if chosen as the next edge $e_{i+1}$. 
Then 
\begin{equation}
\label{eq:prob_no_closed_edge_0}
\begin{split}
\PP[\cE_{F,m} \wedge \cH_m] &= \PP[\cE_{F,m/2} \wedge \cH_{m/2}] \prod_{m/2 \leq i \leq m-1} \PP[\cE_{F,i+1} \wedge \cH_{i+1} \mid \cE_{F,i} \wedge \cH_i]\\
&\leq\prod_{m/2 \leq i \leq m-1} \PP[e_{i+1} \notin O_{F}(i) \mid \cE_{F,i} \wedge \cH_i] \enspace .
\end{split}
\end{equation}
Note that $\cE_{F,i} \wedge \cH_i$ depends only on the first $i$ steps, so given this, the process fails to choose $e_{i+1}$ from $O_{F}(i)$ with probability $1-|O_{F}(i)|/|O(i)|$. 
With this in mind, we claim that in order to prove \eqref{eq:prob_closed_small} it suffices to show that for $m/2 \leq i \leq m$, whenever $\cE_{F,i} \wedge \cH_i$ holds we have
\begin{equation}
\label{eq:bound_closed_edge}
|O_{F}(i)| \geq \frac{13a \log n}{m} |O(i)| \enspace .
\end{equation}
Indeed, combining \eqref{eq:prob_no_closed_edge_0} and \eqref{eq:bound_closed_edge}, using the inequality $1-x \leq e^{-x}$ we deduce, say, 
\begin{equation}
\label{eq:prob_no_closed_edge}
\PP[\cE_{F,m} \wedge \cH_m] \leq e^{-6a \log n} =  n^{-6a} \enspace . 
\end{equation}
Now, taking a union bound over all choices of $A$ and $F$, we obtain
\begin{equation*}
\PP[\cE_m \wedge \cH_m] \leq \sum_{1 \leq a \leq n^{d}} \binom{n}{a} \binom{\binom{a}{2}}{\lceil c a \rceil} n^{-6a} \leq \sum_{1 \leq a \leq n^{d}} n^{a} a^{2(ca+1)} n^{-6a} \enspace . 
\end{equation*}  
Using $1 \leq a \leq n^{d}$ and \eqref{eq:def:c:d}, i.e.\ $d \leq \min\{1/c,1\}$,  we see that
\[
n^{a} a^{2(ca+1)} n^{-6a} \leq n^{2(ca+1)d} \cdot n^{-5a} \leq n^{2a+2} \cdot n^{-5a} \leq n^{-a} \enspace , 
\]
which readily implies $\PP[\cE_m \wedge \cH_m] = o(1)$. 
To sum up, assuming \eqref{eq:bound_closed_edge} we have established the desired formula \eqref{eq:prob_closed_small}.

In the remainder we prove \eqref{eq:bound_closed_edge} for $m/2 \leq i \leq m$, whenever $\cE_{F,i} \wedge \cH_i$ holds. 
Since $uv \in F \cap O(i)$ would belong to $C(i+1)$ iff $e_{i+1} \in C_{uv}(i)$, we deduce $O_{F}(i) = \bigcup_{uv \in F \cap O(i)} C_{uv}(i)$. 
Therefore 
\begin{equation}
\label{eq:bound_closed_edge_1}
|O_F(i)| \geq \sum_{uv \in F \cap O(i)} |C_{uv}(i)| \ \ - \sum_{\substack{uv,u'v' \in F \cap O(i)\\ uv \neq u'v'}} |C_{uv}(i) \cap C_{u'v'}(i)|
 \enspace .
\end{equation}
Observe that for $a \leq n^{d}$ we have $a \geq pa^2n^{2\epsilon}$ by definition of $p$ and $d$, cf.\ \eqref{eq:parameters:functions} and \eqref{eq:def:c:d}. 
Recall that \eqref{eq:edges-estimate} holds on $\cH_i$. 
So, using $|F| = \lceil ca \rceil$ and \eqref{eq:def:c:d}, i.e.\ $c \geq 16\epsilon^{-1}$, we deduce 
\begin{equation}
\label{eq:edges-estimate:A}
e(A) \leq \max\left\{8\epsilon^{-1}a, pa^2n^{2\epsilon} \right\} = 8\epsilon^{-1} a \leq c/2 \cdot a \leq |F|/2 \enspace \enspace . 
\end{equation} 
Recall that whenever $\cE_{F,i}$ holds, then we have $F \cap C(i) = \emptyset$, which in turn implies $F \cap O(i) = F \setminus E(i)$. 
So, using $F \subseteq \binom{A}{2}$ and \eqref{eq:edges-estimate:A}, we see that  
\begin{equation}
\label{eq:open:pairs:estimate:F}
|F \cap O(i) | = |F \setminus E(i)| \geq |F| - e(A) \geq |F|/2 \enspace .
\end{equation} 
Note that \eqref{eq:closed-estimate} and \eqref{eq:closed-intersection-estimate} hold on $\cH_i$. 
Substituting these estimates as well as \eqref{eq:open:pairs:estimate:F} into \eqref{eq:bound_closed_edge_1}, we have 
\begin{equation*}
\begin{split}
|O_F(i)| &\geq  |F|/2 \cdot \beta_H (2t)^{e_H-2}q(t)p^{-1} - |F|^2 \cdot n^{-1/e_H}p^{-1} \enspace . 
\end{split}
\end{equation*}
Observe that \eqref{eq:def:eps:mu} and \eqref{eq:parameters:functions} imply $q(t) \geq n^{-\epsilon}$ for $i \leq m$. 
So, since $|F| \leq 3cn^{d}$, by \eqref{eq:def:c:d} we see that $|F|n^{-1/e_H} \leq 3c n^{-\epsilon} \leq 3c q(t)$. 
Using that $t = i/(n^{2}p) = \omega(1)$ for $i \geq m/2$, we crudely obtain 
\begin{equation*}
|O_F(i)| \geq |F| \beta_H 2^{e_H-4} \cdot t^{e_H-2} q(t)p^{-1} \enspace . 
\end{equation*}
Note that on $\cH_i$ we furthermore have $q(t) \geq |O(i)|/n^2$ by \eqref{eq:open-estimate}. 
So, writing $t = i/(n^{2}p)$ and using $|F| \geq ca$ as well as \eqref{eq:def:c:d}, for $m/2 \leq i \leq m = \mu n^2 p (\log n)^{1/(e_H-1)} $ we deduce
\begin{equation*}
\begin{split}
|O_F(i)|  & \geq |F| \beta_H 2^{e_H-4}   \frac{i^{e_H-2}}{(n^{2}p)^{e_H-2}p} q(t) \geq |F| \beta_H 2^{e_H-4}  \frac{i^{e_H-1}}{i(n^{2}p)^{e_H-1}} |O(i)| \\
& \geq |F| \frac{\beta_H\mu^{e_H-1}}{2^{5}} \cdot \frac{\log n}{m} |O(i)| \geq  \frac{c\beta_H\mu^{e_H-1}}{2^{5}} \cdot  \frac{a\log n}{m} |O(i)| \geq \frac{13a\log n}{m} |O(i)|\enspace . 
\end{split}
\end{equation*}
To summarize, we have established \eqref{eq:bound_closed_edge} and, as explained, this completes the proof. 
\end{proof}
The main difficulty in the above prove is the estimate \eqref{eq:bound_closed_edge}. 
A similar bound is implicit in the approach of Gerke and Makai~\cite{GerkeMakai2010K3}, but their argument is tailored towards the (simpler) $C_3$-free case. 
By contrast, our proof exploits a combinatorial characterization of $O_{F}(i)$ for the more general $H$-free process, where $H$ is strictly $2$-balanced,  which in turn enables us to prove stronger results (e.g.\ we allow for $|A| \leq n^d$ instead of constant size). 
In fact, we can also obtain the asymptotic size of $|O_F(i)|$ using the results in \cite{BohmanKeevash2010H}; we leave these details to the interested reader. 
Furthermore, our approach avoids a subtle conditioning issue: in~\cite{GerkeMakai2010K3} the authors condition on events that depend on the first $m$ steps, but seem to assume that after the first $i$ steps, with $i < m$, the next edge $e_{i+1}$ is still chosen uniformly at random from $O(i)$.

\bigskip{\bf Acknowledgements.} 
I am grateful to my supervisor Oliver Riordan for carefully reading an earlier version of this paper and for helpful comments.
Furthermore, I would like to thank Angelika Steger for drawing my attention to the triangle-free process, and, in particular, to a preprint of \cite{Bohman2009K3}.

\small
\bibliographystyle{plain}

\end{document}